\documentclass[12pt]{amsart}

\usepackage[utf8]{inputenc}
\usepackage[margin=1in]{geometry}  
\usepackage{graphicx}              
\usepackage{amsmath}               
\usepackage{amsfonts}              
\usepackage{amsthm}                
\usepackage{caption}
\usepackage{subcaption}
\usepackage{ amssymb }
\usepackage[shortlabels]{enumitem}	
\usepackage{mathtools}				
\usepackage{hyperref}				
\usepackage{tikz-cd}    			
\usepackage{ytableau}
\usepackage{multirow}
\usepackage{tikz}
\usepackage{ytableau}
\usepackage{blkarray}

\newtheorem{thm}{Theorem}[section]
\newtheorem{lem}[thm]{Lemma}
\newtheorem{prop}[thm]{Proposition}
\newtheorem{cor}[thm]{Corollary}

\newtheorem{de}[thm]{Definition}

\theoremstyle{definition}
\newtheorem{rmk}[thm]{Remark}
\newtheorem{example}[thm]{Example}

\newcommand{\R}{\mathbb{R}}      			
\newcommand{\Z}{\mathbb{Z}}      			
\newcommand{\C}{\mathbb{C}}      			
\newcommand{\Q}{\mathbb{Q}}      			
\newcommand{\im}{\operatorname{Im}}			
\renewcommand{\ker}{\operatorname{Ker}}		
\newcommand{\rk}{\operatorname{Rk}}       
\newcommand{\trop}{\operatorname{Trop}}
\newcommand{\gl}{\operatorname{GL}}
\newcommand{\ind}{\operatorname{Ind}}
\newcommand{\val}{\operatorname{val}}

\newcommand{\sgn}{\operatorname{Sgn}}
\newcommand{\lie}{\operatorname{Lie}}
\newcommand{\ob}{\operatorname{Ob}}
\newcommand{\G}{\textbf{G}}

\newcommand{\thetaThreeOneOne}[5]{
\begin{tikzpicture}[baseline=-1, line cap=round,line join=round,x=1cm,y=1cm]
        \draw [line width=1.pt] (-0.5,0) -- (0.5,0);
        \draw [fill=black] (-0.5,0) circle (1pt);
        \draw [fill=black] (0.5,0) circle (1pt);
        \draw [line width=1.pt] (-0.5,0) .. controls (-0.4,0.8) and (0.4, 0.8) .. (0.5,0);
        \draw [line width=1.pt] (-0.5,0) .. controls (-0.4,-0.8) and (0.4, -0.8) .. (0.5,0);
        \draw [fill=black] (0,0.6) circle (1pt);
        \draw [fill=black] (0,0) circle (1pt);
        \draw [fill=black] (0,-0.6) circle (1pt);
        \draw [fill=black] (-0.25,0.5) circle (1pt);
        \draw [fill=black] (0.25,0.5) circle (1pt);
        \draw [line width=1.pt] (-0.25,0.5) -- (-0.3,0.8);
        \draw (-0.3,0.8) node[anchor=south] {#1};
        \draw [line width=1.pt] (0,0.6) -- (0,0.9);
        \draw (0,0.9) node[anchor=south] {#2};
        \draw [line width=1.pt] (0.25,0.5) -- (0.3,0.8);
        \draw (0.3,0.8) node[anchor=south] {#3};
        \draw [line width=1.pt] (0,0) -- (0,0.3);
        \draw (0,0.3) node[anchor= west] {#4};
        \draw [line width=1.pt] (0,-0.6) -- (0,-0.3);
        \draw (0,-0.3) node[anchor= west] {#5};
        \end{tikzpicture}}
\newcommand{\thetaTwoTwoOne}[5]{
\begin{tikzpicture}[baseline=-1, line cap=round,line join=round,x=1cm,y=1cm]
        \draw [line width=1.pt] (-0.5,0) -- (0.5,0);
        \draw [fill=black] (-0.5,0) circle (1pt);
        \draw [fill=black] (0.5,0) circle (1pt);
        \draw [line width=1.pt] (-0.5,0) .. controls (-0.4,0.8) and (0.4, 0.8) .. (0.5,0);
        \draw [line width=1.pt] (-0.5,0) .. controls (-0.4,-0.8) and (0.4, -0.8) .. (0.5,0);
        \draw [fill=black] (-0.2,0.55) circle (1pt);
        \draw [fill=black] (0.2,0.55) circle (1pt);
        \draw [fill=black] (-0.2,0) circle (1pt);
        \draw [fill=black] (0.2,0) circle (1pt);
        \draw [fill=black] (0,-0.6) circle (1pt);
        \draw [line width=1.pt] (-0.2,0.55) -- (-0.2,0.85);
        \draw (-0.2,0.85) node[anchor=south] {#1};
        \draw [line width=1.pt] (0.2,0.55) -- (0.2,0.85);
        \draw (0.2,0.85) node[anchor=south] {#2};
        \draw [line width=1.pt] (-0.2,0) -- (-0.2,0.2);
        \draw (-0.2,0.14) node[anchor=south] {#3};
        \draw [line width=1.pt] (0.2,0) -- (0.2,0.2);
        \draw (0.2,0.14) node[anchor=south] {#4};
        \draw [line width=1.pt] (0,-0.6) -- (0,-0.3);
        \draw (0,-0.3) node[anchor= west] {#5};
        \end{tikzpicture}
}
\newcommand{\thetaTwoOneOne}[5]{
\begin{tikzpicture}[baseline=-1, line cap=round,line join=round,x=1cm,y=1cm]
        \draw [line width=1.pt] (-0.5,0) -- (0.5,0);
        \draw [fill=black] (-0.5,0) circle (1pt);
        \draw [fill=black] (0.5,0) circle (1pt);
        \draw [line width=1.pt] (-0.5,0) .. controls (-0.4,0.8) and (0.4, 0.8) .. (0.5,0);
        \draw [line width=1.pt] (-0.5,0) .. controls (-0.4,-0.8) and (0.4, -0.8) .. (0.5,0);
        \draw [fill=black] (-0.2,0.55) circle (1pt);
        \draw [fill=black] (0.2,0.55) circle (1pt);
        \draw [fill=black] (0,0) circle (1pt);
        \draw [fill=black] (0,-0.6) circle (1pt);
        \draw [line width=1.pt] (-0.2,0.55) -- (-0.2,0.85);
        \draw (-0.2,0.85) node[anchor=south] {#1};
        \draw [line width=1.pt] (0.2,0.55) -- (0.2,0.85);
        \draw (0.2,0.85) node[anchor=south] {#2};
        \draw [line width=1.pt] (0,0) -- (0,0.2);
        \draw (0,0.2) node[anchor=south] {#3};
        \draw [line width=1.pt] (0,-0.6) -- (0,-0.3);
        \draw (0,-0.3) node[anchor= west] {#4};
        \draw [line width=1.pt] (0.5,0) -- (0.7,0);
        \draw (0.7,0) node[anchor= west] {#5};
        \end{tikzpicture}}

\title{The $S_n$-equivariant rational homology of the tropical moduli spaces $\Delta_{2,n}$}
\author[C.H. Yun]{Claudia He Yun}\address{Department of Mathematics, Brown University, Box
1917, Providence, RI 02912}\email{he\_yun@brown.edu}
\date{}

\begin{document}

\maketitle

\begin{abstract}
    We compute the $S_n$-equivariant rational homology of the tropical moduli spaces $\Delta_{2,n}$ for $n\leq 8$ using a cellular chain complex for symmetric $\Delta$-complexes in Sage.
\end{abstract}

\section{Introduction}

The tropical moduli space $\Delta_{g,n}$ is a topological space that parametrizes isomorphism classes of $n$-marked stable tropical curves of genus $g$ with total volume 1. This space has been studied for its own sake in tropical geometry, and has connections to several other spaces of interest. For example, $\Delta_{g,n}$ is identified with the link of the vertex in the tropical moduli space $M_{g,n}^{\trop}$ as well as the boundary complex of $\overline{\mathcal{M}}_{g,n}$, the algebraic moduli space of stable curves of genus $g$ and $n$ markings \cite{abramovich2015tropicalization, CGP}. The symmetric group naturally acts on $\Delta_{g,n}$ by permuting its marked points and thus induces the structure of $S_n$-representations on its reduced rational homology.

In this paper, we compute the $S_n$-equivariant rational homology of $\Delta_{2,n}$ for $n$ up to 8 in SageMath.

\begin{thm}
The characters of the $S_n$-equivariant rational homology groups $\tilde{H_i}(\Delta_{2,n};\Q)$ for $n$ up to 8 are given in Table \ref{decomposition} as sums of irreducible characters.
\label{mainThm}
\end{thm}

\begin{table}[]
\begin{tabular}{|l|l|}
\hline
$n$ & $S_n$-equivariant homology \\ \hline
\multirow{2}{*}{0} & $\chi(\tilde{H}_{2}(\Delta_{2,0};\Q)) = 0$  \\
\cline{2-2} & $\chi(\tilde{H}_{1}(\Delta_{2,0};\Q)) = 0$  \\ \hline
\multirow{2}{*}{1} & $\chi(\tilde{H}_{3}(\Delta_{2,1};\Q)) = 0$  \\
\cline{2-2} & $\chi(\tilde{H}_{2}(\Delta_{2,1};\Q)) = 0$  \\ \hline
\multirow{2}{*}{2} & $\chi(\tilde{H}_{4}(\Delta_{2,2};\Q)) = \chi_{2}$  \\
\cline{2-2} & $\chi(\tilde{H}_{3}(\Delta_{2,2};\Q)) = 0$  \\ \hline
\multirow{2}{*}{3} & $\chi(\tilde{H}_{5}(\Delta_{2,3};\Q)) = 0$  \\
\cline{2-2} & $\chi(\tilde{H}_{4}(\Delta_{2,3};\Q)) = 0$  \\ \hline
\multirow{2}{*}{4} & $\chi(\tilde{H}_{6}(\Delta_{2,4};\Q)) = \chi_{211}$  \\
\cline{2-2} & $\chi(\tilde{H}_{5}(\Delta_{2,4};\Q)) = \chi_{4}$  \\ \hline
\multirow{2}{*}{5} & $\chi(\tilde{H}_{7}(\Delta_{2,5};\Q)) = \chi_{41}+\chi_{32}+\chi_{311}$ \\ 
\cline{2-2} & $\chi(\tilde{H}_{6}(\Delta_{2,5};\Q)) = \chi_{32}$ \\ \hline
\multirow{2}{*}{6} & $\chi(\tilde{H}_{8}(\Delta_{2,6};\Q)) = \chi_{6}+\chi_{1^6}+\chi_{51}+\chi_{21111}+2\chi_{33}+\chi_{222}+2\chi_{42}+\chi_{2211}+2\chi_{321}$ \\ 
\cline{2-2} & $\chi(\tilde{H}_{7}(\Delta_{2,6};\Q)) = \chi_{411}+\chi_{321}$ \\ \hline
\multirow{2}{*}{7} & $\chi(\tilde{H}_{9}(\Delta_{2,7};\Q)) = \chi_{2221}+3\chi_{31111}+4\chi_{3211}+3\chi_{322}+\chi_{331}+3\chi_{4111}+5\chi_{421}+\chi_{43}$ \\
& $+\chi_{511}+2\chi_{52}$ \\ 
\cline{2-2} & $\chi(\tilde{H}_{8}(\Delta_{2,7};\Q)) = \chi_{1^7}+\chi_{2221}+\chi_{3211}+\chi_{331}+\chi_{4111}+\chi_{421}+\chi_{43}+\chi_{511}$ \\ \hline
\multirow{3}{*}{8} & $\chi(\tilde{H}_{10}(\Delta_{2,8};\Q)) = \chi_{1^8}+2\chi_{21^6}+3\chi_{221^4}+5\chi_{2^311}+\chi_{2^4}+7\chi_{32111}+6\chi_{3221}+9\chi_{3311}$\\
& $+5\chi_{332}+\chi_{41^4} + 10\chi_{4211} + 2\chi_{422} + 10\chi_{431} + \chi_{44} + 7\chi_{5111} + 7\chi_{521} + 2\chi_{53} + 5\chi_{611}$\\
\cline{2-2} & $\chi(\tilde{H}_{9}(\Delta_{2,8};\Q)) = \chi_{31^5} + 2\chi_{32111} + 2\chi_{3221} + \chi_{3311} + \chi_{332} +2\chi_{4211} + 2\chi_{422} + 2\chi_{431} $ \\
& $+ \chi_{5111}+2\chi_{521} + \chi_{53} +\chi_{62} + \chi_8$ \\ \hline
\end{tabular}
\caption{Decompositions of the characters of the $S_n$-representations afforded by $\tilde{H}_i(\Delta_{2,n};\Q).$ Irreducible characters are indexed by partitions of $n$. Explicit characters are in Appendix \ref{charTable}.}
\label{decomposition}
\end{table}


The $S_n$-equivariant rational homology of $\Delta_{g,n}$ is known for $g=0$ and $g=1$. When $g=0$ and $n\geq 4$, $\Delta_{0,n}$ is the space of fully grown trees $T_{n-1}$ studied by Robinson and Whitehouse \cite{Robinson1996TheTR} and the space of phylogenetic trees with $n$ leaves by Billera, Holmes and Vogtmann \cite{BILLERA2001733}. \cite{Robinson1996TheTR} shows that $T_n$ is homotopy equivalent to the wedge of $(n-1)!$ spheres of dimension $n-3$ and they calculate the character of the integral $S_{n+1}$-representation carried by $H_{n-3}(T_n;\Z)$ to be \[\epsilon\cdot(\ind_{S_{n}}^{S_n+1}\lie_{n} - \lie_{n+1})\] where $\epsilon$ is the alternating character and $\lie_r = \ind_{C_r}^{S_r}\rho_r$ \cite{brandt1944}. Here $\rho_r$ is a faithful linear character of the subgroup in $S_r$ generated by an $r$-cycle. The restriction of this representation to $S_n$ coincides with a virtual representation of $S_n$ on the reduced complex homology groups of the order complex of the partition lattice $\Pi_n$ that is studied by Stanley \cite{STANLEY1982132}. 

When $g=1$, Chan, Galatius, and Payne show that $\Delta_{1,n}$ is a wedge of $(n-1)!/2$ spheres of dimension $n-1$ and the representation of $S_n$ on $H_{n-1}(\Delta_{1,n};\Q)$ is \[\text{Ind}_{D_{n,\phi}}^{S_n}\text{Res}_{D_{n,\psi}}^{S_n}\text{sgn}\] where $\phi$ is the action of the dihedral group on the vertices of an $n$-gon and $\psi$ is the action on the edges \cite{CGP,CGPTopOfMwithMarkedPoints}. This result can also be derived from Geztler's computation of the $S_n$-equivariant Serre characteristic of $\mathcal{M}_{1,n}$ \cite{getzler1999}. 

For higher genera, the homotopy type of $\Delta_{g,n}$ is only known for a few pairs of $(g,n)$. In \cite{allcock2019trop}, it is shown that $\Delta_{g,n}$ is simply connected for $g\geq 1$ and that the space $\Delta_{3,0}$ is homotopy equivalent to $S^5$. There have also been calculations of the homology and cohomology of $\Delta_{g,n}$. For $\Delta_{2,n}$, Chan calculates its rational homology for $n$ up to 8 \cite{ChanTopOfM2n}. This paper extends Chan's work by promoting \cite[Table 1]{ChanTopOfM2n} to be $S_n$-representations. In particular, we may recover \cite[Table 1]{ChanTopOfM2n} by evaluating the character functions in Table \ref{decomposition} at the identity element.

The tropical moduli space $\Delta_{g,n}$ provides information on the boundary complex of the Deligne-Mumford compactification $\overline{\mathcal{M}}_{g,n}$ of $\mathcal{M}_{g,n}$ by stable curves. By the work of Deligne \cite{PMIHES_1971__40__5_0, PMIHES_1974__44__5_0}, there is a natural isomorphism 
\begin{equation}
    \tilde{H}_{k-1}(\Delta_{g,n};\Q) \cong \text{Gr}_{6g-6+2n}^WH^{6g-6+2n-k}(\mathcal{M}_{g,n};\Q)
    \label{Top weight cohomology of Mgn is identified with homology of Deltagn}
\end{equation} that identifies the reduced rational homology of $\Delta_{g,n}$ with the top weight rational cohomology of $\mathcal{M}_{g,n}$. Specializing to the case $g=2$, we obtain the following corollary.

\begin{cor}
The top weight rational cohomology $\text{Gr}_{6+2n}^WH^{i}(\mathcal{M}_{2,n};\Q)$ is supported in degrees $n+3$ and $n+4$. The representations of $S_n$ on the groups $\text{Gr}_{6+2n}^WH^{n+3}(\mathcal{M}_{2,n};\Q)$ and $\text{Gr}_{6+2n}^WH^{n+4}(\mathcal{M}_{2,n};\Q)$, induced by permuting marked points, are given in Table \ref{decomposition}.
\label{mainCor}
\end{cor}


There have been several previous studies on the $S_n$-equivariant invariants of $\mathcal{M}_{g,n}$ and related moduli spaces. Gorsky calculates the $S_n$-equivariant Euler characteristics of $\mathcal{M}_{2,n}$ \cite{gorsky2007snequivariant}. In a later paper, he derives the generating function for the $S_n$-equivariant Euler characteristics of $\mathcal{M}_{g,n}$ \cite{gorsky2009equivariant}. He also computes the generating function for the $S_n$-equivariant Euler characteristics of $\mathcal{H}_{g,n}$, the moduli space of hyperelliptic curves of genus $g$ with $n$ marked points, for $g \geq 2$ \cite{gorsky2008snequivariant}. Recently, \cite{CFGP} gives the generating function for the $S_n$-equivariant top weight Euler characteristics of $\mathcal{M}_{g,n}$, which, by the identification of Equation \eqref{Top weight cohomology of Mgn is identified with homology of Deltagn}, gives the generating function for the $S_n$-equivariant Euler characteristics of $\Delta_{g,n}$. Note that since $\Delta_{0,n}$ and $\Delta_{1,n}$ both have homotopy types of a wedge of spheres, calculating the $S_n$-equivariant Euler characteristics is exactly calculating the characters of the $S_n$-representations afforded by their only nontrivial reduced homology group. The case when $g=2$ is the first instance where individual characters afforded by homology groups are opaque. See Remark \ref{discussionOfFarber} for more discussion on the connection between the computations of this paper and the Euler characteristics result of \cite{CFGP}.

The organization of this paper is as follows. In Section \ref{section: prelim}, we review the construction of the tropical moduli space $\Delta_{g,n}$ for general $g$ and $n$. We also recall the definitions of symmetric $\Delta$-complexes and a particular cellular chain complex that computes the rational homology of its geometric realization. In Section \ref{computation}, we specialize to the genus 2 case and compute $\tilde{H}_*(\Delta_{2,n};\Q)$. We also discuss the performance of the Sage program we use for this computation and analyze the combinatorial structure of a particular subrepresentation of $\tilde{H}_7(\Delta_{2,5};\Q)$.

\vspace{2mm} 

\noindent \textbf{Acknowledgments.}
I am grateful to Melody Chan for suggesting this direction of research, answering numerous questions, providing useful feedback on my writing, and encouraging me in times of doubt. I thank Christin Bibby and Nir Gadish for enlightening discussions on the topic. I also thank Sarah Griffith for many words of encouragement and mathematical discussions.

\section{Preliminaries}
\label{section: prelim}
In this section we review the general theory of tropical curves and their moduli spaces, following Section 2 in \cite{CGPTopOfMwithMarkedPoints} and Section 2 in \cite{ChanTopOfM2n}. We also recall the notions of symmetric $\Delta$-complexes and a cellular chain complex that calculates their geometry realizations, following Section 3 in \cite{CGPTopOfMwithMarkedPoints}.
\subsection{The tropical moduli spaces $\Delta_{g,n}$}
\label{section: delta_gn}
All graphs in consideration are finite and connected multigraphs, where loops and parallel edges are allowed. Given a graph $G$, we denote the edge set and the vertex set by $E(G)$ and $V(G)$, respectively. The valence $\val(v)$ of a vertex $v$ of $\textbf{G}$ is the number of half-edges incident to $v$. A loop at a vertex contributes twice to the valence.

A weight function $w$ on a graph $G$ is a function $w: V(G) \to \Z_{\geq 0}$. A pair $(G,w)$ is called a weighted graph. The genus of such a graph is \[g(G,w):= b_1(G) + \sum_{v \in V(G)} w(v),\] where $b_1(G) = |E(G)|-|V(G)|+1$ is the first Betti number of $G$.

An $n$-marking on a graph $G$ is a function $m:\{1,2,\cdots,n\} \to V(G)$. An $n$-marked weighted graph $\textbf{G}$ is therefore a triple $\textbf{G}=(G,w,m)$ where $(G,w)$ is a weighted graph and $m$ is an $n$-marking function. Such a graph is stable if the following condition, called the stability condition, holds: for every vertex $v \in V(G)$, we have \[2w(v)-2+\val(v)+|m^{-1}(v)| > 0.\]

Now we recall the category $\Gamma_{g,n}$. It is described in detail in \cite{CGPTopOfMwithMarkedPoints}. The objects in this category are connected stable $n$-marked weighted graphs of genus $g$. For $e\in E(\textbf{G})$ that is not a loop, there is a morphism from $\textbf{G}$ to the edge contracted graph $\G' = \textbf{G}/e$. The graph $\G'$ is obtained by removing $e$ and identifying its endpoints $\{u,v\}$ as one vertex $[e]$. The new vertex $[e]$ satisfies $w'([e]) = w(u)+w(v)$ and $m'^{-1}([e]) = m^{-1}(u) \cup m^{-1}(v)$. If $e$ is a loop based at $v$, then $\textbf{G}/e$ is obtained by removing $e$ and increasing the weight of $v$ by 1. The morphisms in $\Gamma_{g,n}$ are compositions of edge contractions $\textbf{G} \to \textbf{G}/e$ and isomorphisms $\textbf{G} \to \textbf{G}'$.

Two graphs $\textbf{G}$ and $\textbf{G}'$ are said to have the same combinatorial type if they are isomorphic in $\Gamma_{g,n}$. For our purposes, we may replace $\Gamma_{g,n}$ with a skeleton subcategory, by choosing a representative for each isomorphism class. That makes $\Gamma_{g,n}$ a finite category.

Following the language in \cite{ChanTopOfM2n}, for $g \geq 2$, we define the unmarked type of $\textbf{G} = (G,m,w)$ to be the isomorphism class of $(G',w') \in \Gamma_{g,0}$ obtained in the follow way: $G'$ is the smallest connected subgraph of $G$ that contains all cycles of $G$ and positively weighted vertices. We suppress vertices of valence 2 so that $(G',w')$ is stable.

A length function $l$ on $\textbf{G}=(G,w,m) \in \Gamma_{g,n}$ is an element in $\R^{E(G)}_{>0}$. It assigns a positive real number, i.e., length, to each edge in $G$. Two pairs $(\textbf{G},l)$ and $(\textbf{G}',l')$ are isometric if there is an isomorphism $\phi: \textbf{G} \to \textbf{G}'$ in $\Gamma_{g,n}$ for which 
\[l(e) = l'(\phi(e)) \quad \text{ for all } e\in E(\textbf{G}).\] The pair $(\textbf{G},l)$ is called a (stable) tropical curve and the sum of all edge lengths is called the volume.

Now we are in a position to construct the moduli space of tropical curves $M^{\text{trop}}_{g,n}$. The construction and notations follow \cite{ChanTopOfM2n, CGP}. Define
\begin{align*}
    \overline{C}(\textbf{G}) &:= \R^{E(G)}_{\geq 0},
\end{align*}
where $\textbf{G} \in \Gamma_{g,n}$. Given two stable weighted graphs $\textbf{G}$ and $\textbf{G}'$, and an isomorphism in $\Gamma_{g,n}$ \[\alpha: \textbf{G}' \to \textbf{G}/S\] for some $S\subset E(\textbf{G})$, we define a linear map \[L_\alpha: \overline{C}(\textbf{G}') \to \overline{C}(\textbf{G})\] that sends an element $l'\in \overline{C}(\textbf{G}')$ to $l\in \overline{C}(\textbf{G})$ with \[l(e) = \begin{cases}
l'(\alpha^{-1}(e)) & e \not \in S \\
0 & e \in S
\end{cases}
\] for all $e \in E(\textbf{G})$. The map $L_\alpha$ identifies $\overline{C}(\textbf{G}')$ with the face in $\overline{C}(\textbf{G})$ where all edges in $S$ have 0 length.

\begin{de}
The tropical moduli space $M^{\text{trop}}_{g,n}$ is the colimit in the category of topological spaces \[\varinjlim(\{\overline{C}(\textbf{G})\},\{L_{\alpha}\})\] where $\textbf{G}$ ranges over all graphs in $\Gamma_{g,n}$ and $L_{\alpha}$ are all linear maps defined as above.
\end{de}
\noindent The volume of a tropical curve extends to a function $v: M^{\text{trop}}_{g,n} \to \R_{\geq 0}$.
\begin{de}
The subspace $\Delta_{g,n}$ is the inverse image $v^{-1}(1)$ of the volume function at 1.
\end{de}
\noindent Therefore $\Delta_{g,n}$ parametrizes tropical curves of volume 1.

\subsection{Symmetric $\Delta$-complexes}
\label{section: sym delta complex}

Denote the set $\{0,1,\cdots,p\}$ by $[p]$.

\begin{de}[\cite{CGP}]
Let $I$ be the category whose objects are $[p]$ for each $p \geq 0$ together with $[-1] := \emptyset$ and whose morphisms are all injective maps. A symmetric $\Delta$-complex is a functor $X: I^{\text{op}} \to \text{Sets}$.
\end{de}

\begin{rmk}
Let $\delta^i: [p-1] \to [p]$ be the unique order-preserving injection whose image misses $i$ for $0\leq i\leq p$. Morphisms in $I$ are generated by permutations $[p] \to [p]$ and the $\delta^i$.
\end{rmk}

Recall that the standard $p$-simplex $\Delta^p \subset \R^{p+1}$ is defined to be the convex hull of the standard basis vectors $e_0,\dots,e_p$, i.e., \[\Delta^p = \left\{\sum_{i=0}^p t_ie_i : \sum t_i = 1 \text{ and } t_i \geq 0 \text{ for all }i\right\}.\] Given a set map $\theta: [p] \to [q]$, it induces a map between simplices $\theta_*:\Delta^p \to \Delta^q$ where \[\theta_*\left(\sum_{i=0}^p t_ie_i\right) = \sum_{i=0}^q \left(\sum_{j \in \theta^{-1}(i)}t_j\right)e_i.\]

For notational simplicity, we denote $X([p])$ by $X_p$. The geometric realization of a symmetric $\Delta$-complex $X$ is the topological space \[|X| = \left(\coprod_p X_p \times \Delta^p\right)/\sim,\] where the equivalence relation is given by $(x,\theta_*a) \sim (\theta^*x,a)$, where $x\in X_p$, $\theta: [p] \to [q]$ injective, and $a \in \Delta^p$. This equivalence relation is analogous to the gluing maps of a $\Delta$-complex.

\begin{prop}\cite{CGPTopOfMwithMarkedPoints}
The moduli space $\Delta_{g,n}$ is the geometric realization of a symmetric $\Delta$-complex.
\end{prop}

\noindent By abuse of notation, we use $\Delta_{g,n}$ to refer to both the functor and its geometric realization.
\begin{rmk}
We describe explicitly what the functor $\Delta_{g,n}$ does on objects and morphisms. The elements in the set $\Delta_{g,n}([p])$ are equivalence classes of $(\textbf{G},\tau)$ where $\textbf{G} \in \ob(\Gamma_{g,n})$ has exactly $p+1$ edges and $\tau: E(G) \to [p]$ is a bijective function, called an edge-labeling. Two edge-labelings on $\textbf{G}$ are equivalent if they are related by an automorphism of $\textbf{G}$. 

On morphisms, $\Delta_{g,n}$ sends an injective map $\theta:[q] \to [p]$ to a set map $f_\theta: \Delta_{g,n}([p]) \to \Delta_{g,n}([q])$ defined as follows. Let $(\textbf{G},\tau) \in \Delta_{g,n}([p])$. Then $f_\theta((\textbf{G},\tau)) = (\textbf{G}/S,\tau')$ where $S = \{e \in E(\textbf{G}) : \tau(e) \in [p]\backslash \im \theta\}$ and $\tau'(e) = \theta^{-1}(\tau(e))$ for $e \in E(\textbf{G}/S)$.
\end{rmk}

A subcomplex of a symmetric $\Delta$-complex is a subfunctor. The following definition allows us to produce a subcomplex of $\Delta_{g,n}$ by specifying simplices. 

\begin{de}
Let $S$ be a subset of $\ob(\Gamma_{g,n})$ and $\overline{S}$ be its closure with respect to edge contraction. Then we define the symmetric $\Delta$-complex generated by $S$ to be the following functor $\Delta_{g,n}^S$.
\begin{itemize}
    \item $\Delta_{g,n}^S$ sends an object $[p]$ to the set of equivalence classes of graphs $(\textbf{G}, \tau) \in \Delta_{g,n}([p])$ for all $\textbf{G}\in \overline{S}\subset \ob(\Gamma_{g,n})$.
    \item $\Delta_{g,n}^S$ sends a morphism $\theta:[p] \to [q]$ to the restriction of $\Delta_{g,n}(\theta)$ to $\Delta_{g,n}^S([p])$.
\end{itemize}
\label{subcomplex}
\end{de}
\noindent This is a subfunctor of $\Delta_{g,n}$ and therefore a subcomplex.

\subsection{Cellular chain complex for symmetric $\Delta$-complexes}
\label{section: cellular chain complex}
For a symmetric $\Delta$-complex $X$, \cite{CGP} defines a cellular chain complex that calculates the rational singular homology of the geometric realization $|X|$.

\begin{de}
Let $R$ be a commutative ring and write $RX_p$ for the free $R$-module spanned by the set $X_p$. The group of cellular $p$-chains $C_p(X;R)$ is \[C_p(X;R) = R^{\text{sgn}} \otimes_{RS_{p+1}} RX_p\] where $R^{\text{sgn}}$ denotes the action of $S_{p+1}$ on $R$ via the sign.
\end{de}

The boundary map $\partial:C_p(X;R) \to C_{p-1}(X;R)$ is the unique map that makes the following diagram commute:
\[\begin{tikzcd}
RX_p \arrow[r,"\sum (-1)^i(d_i)_*"] \arrow[d,twoheadrightarrow] & RX_{p-1} \arrow[d,twoheadrightarrow,"\pi"] \\
C_p(X;R) \arrow[r,"\partial"] & C_{p-1}(X;R)
\end{tikzcd}\]

\begin{rmk}
We describe a basis for the chain group $C_p(\Delta_{g,n};\Q)$. This group is the quotient of the free $\Q$-vector space on the set $\{(\G,\tau):\G \in \ob(\Gamma_{g,n}), |E(\G)| = p+1, \tau: E(\G) \to [p]\}$ by the following relation. Let $\tau$ and $\tau'$ be two distinct edge-labelings on the same object $\G$. Then set $(\G, \tau) = \sgn(\pi)(\G, \tau')$ where $\pi$ is a permutation on $E(\G)$ such that $\tau = \tau' \circ \pi$. Note that since we replaced $\Gamma_{g,n}$ with a skeleton subcategory by choosing a representative for each isomorphic class of weighted stable graphs, we also have $(\G, \tau) = (\alpha(\G), \tau \circ \alpha) = (\G, \tau \circ \alpha)$ for $\alpha$ any automorphism on $\G$. The parity of $\alpha$ is defined as the parity of the permutation it induces on $E(\G)$. Denote by $[\G,\tau]$ the equivalence class of $(\G,\tau)$ in $C_p(\Delta_{g,n};\Q)$.

Consider the case when $\G$ has an odd automorphism $\alpha$. Then $[\textbf{G},\tau] = [\G,\tau\circ\alpha] = \sgn(\alpha)[\textbf{G},\tau] = -[\textbf{G},\tau]$. So $[\textbf{G},\tau] = 0$.

Therefore, a basis of $C_p(\Delta_{g,n};\Q)$ is \[B = \{[\textbf{G},\tau_\G] : \textbf{G} \in \ob(\Gamma_{g,n}), |E(\textbf{G})| = p+1, \G \text{ has no odd automorphism}\},\] where we pick a representative edge-labeling $\tau_\G$ for each $\G$.
\end{rmk}

\begin{prop}
There is an isomorphism \[H_p(C_*(\Delta_{g,n}; \Q), \partial) \cong \tilde{H}^{\text{sing}}_p(|\Delta_{g,n}|;\Q).\]
\end{prop}

\begin{proof}
$\Delta_{g,n}([-1])$ is the unique object in $\Gamma_{g,n}$ that has no edge. Since it is a singleton, the isomorphism follows from \cite[Proposition 3.9]{CGP}.
\end{proof}

\cite{CGPTopOfMwithMarkedPoints} further introduces a relative cellular homology for a pair of symmetric $\Delta$-complexes. Let $X$ be a subcomplex of a symmetric $\Delta$-complex $Y$ with $\iota:X \to Y$ the inclusion map. Define a relative cellular chain complex $C_*(Y,X;\Q)$ by the short exact sequences \[0 \to C_p(X;\Q) \xrightarrow{\iota_p} C_p(Y;\Q) \to C_p(Y,X;\Q) \to 0 \] for all $p\geq -1$, where $\iota_p$ is the induced inclusion map on the cellular chain groups.

\begin{prop}[\cite{CGPTopOfMwithMarkedPoints}]
\label{prop: relative hom for sym delta complexes}
Let $Y$ be a symmetric $\Delta$-complex and $X \subset Y$ a subcomplex. If $\iota_{-1} : X_{-1} \to Y_{-1}$ is a bijection, then we get a natural isomorphism \[H_p(C_*(Y,X)) \cong H_p(|Y|, |X|;\Q).\]
\end{prop}

\noindent Specializing to our case, we take $Y = \Delta^S_{g,n}$ for some nonempty $S \subset \ob(\Gamma_{g,n})$, as defined in Definition \ref{subcomplex}, and $X = \Delta_{g,n}$. Then Proposition \ref{prop: relative hom for sym delta complexes} implies \[H_p(C_*(\Delta_{g,n},\Delta^S_{g,n})) \cong H_p(|\Delta_{g,n}|,|\Delta^S_{g,n}|;\Q).\]


\section{Computing the homology representation of $\Delta_{2,n}$}
\label{computation}

\subsection{Homology of $\Delta_{2,n}$}

Our goal for this section is to describe the cellular chain complex we use to compute the rational homology of $\Delta_{2,n}$. We identify several subcomplexes with zero reduced rational homology and use the relative chain complex defined in Section \ref{section: cellular chain complex} to simplify the calculation. Some simplifications only apply for $n\geq 4$, but the cases $n < 4$ are small enough to handle by hand. Recall that a bridge in a connected graph is an edge whose removal disconnects the graph.

\begin{de}
The bridge locus $\Delta^{\text{br}}_{g,n}$ in $\Delta_{g,n}$ is the subcomplex generated by all graphs in $\Gamma_{g,n}$ that have bridges, as defined in Definition \ref{subcomplex}.
\end{de}

\begin{rmk}
The set of graphs with bridges is not closed under edge contraction, so the set of simplices of $\Delta^{\text{br}}_{g,n}$ also contains graphs without bridges. By abuse of language, we also call the closure of the set of graphs with bridges under edge contraction in $\ob(\Gamma_{g,n})$ the bridge locus. In Figure \ref{comboTypes}, all combinatorial types but type I are in the bridge locus.
\end{rmk}

\begin{rmk}
The subcomplex $\Delta^{\text{br}}_{g,n}$ is contractible for all $n\geq 0$ by \cite{CGPTopOfMwithMarkedPoints}. Therefore we have an isomorphism $\tilde{H}_p(\Delta_{g,n};\Q) \cong H_p(\Delta_{g,n},\Delta^{\text{br}}_{g,n};\Q)$.
\label{rmk: bridge locus is contractible}
\end{rmk}

\noindent From this point on, we focus on the case when $g=2$.

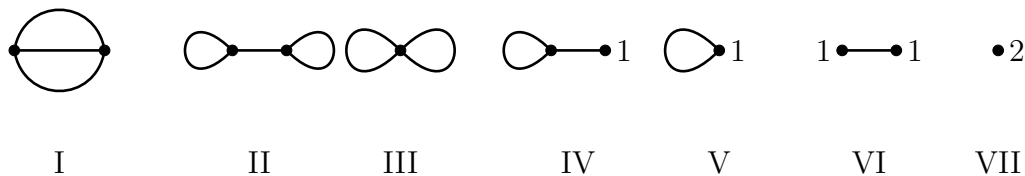
\begin{figure}
    \centering
    \begin{minipage}[t]{.12\textwidth}
        \begin{tikzpicture}[line cap=round,line join=round,x=1.2cm,y=1.2cm]
        \draw [line width=1.pt] (-0.5,0) -- (0.5,0);
        \draw [fill=black] (-0.5,0) circle (2pt);
        \draw [fill=black] (0.5,0) circle (2pt);
        \draw [line width=1.pt] (-0.5,0) .. controls (-0.4,0.6) and (0.4, 0.6) .. (0.5,0);
        \draw [line width=1.pt] (-0.5,0) .. controls (-0.4,-0.6) and (0.4, -0.6) .. (0.5,0);
        \draw (0,-1) node[anchor=north] {I};
        \end{tikzpicture}
    \end{minipage}
    \begin{minipage}[t]{.12\textwidth}
        \begin{tikzpicture}[line cap=round,line join=round,x=1.2cm,y=1.2cm]
        \draw [line width=1.pt] (-0.3,0) -- (0.3,0);
        \draw [fill=black] (-0.3,0) circle (2pt);
        \draw [fill=black] (0.3,0) circle (2pt);
        \draw [line width=1.pt] (-0.3,0) .. controls (-1,0.7) and (-1, -0.7) .. (-0.3,0);
        \draw [line width=1.pt] (0.3,0) .. controls (1,0.7) and (1, -0.7) .. (0.3,0);
        \draw (0,-1) node[anchor=north] {II};
        \end{tikzpicture}
    \end{minipage}
    \begin{minipage}[t]{.12\textwidth}
        \begin{tikzpicture}[line cap=round,line join=round,x=1.2cm,y=1.2cm]
        \draw [fill=black] (0,0) circle (2pt);
        \draw [line width=1.pt] (0,0) .. controls (-0.8,0.8) and (-0.8, -0.8) .. (0,0);
        \draw [line width=1.pt] (0,0) .. controls (0.8,0.8) and (0.8, -0.8) .. (0,0);
        \draw (0,-1) node[anchor=north] {III};
        \end{tikzpicture}
    \end{minipage}
    \begin{minipage}[t]{.12\textwidth}
        \begin{tikzpicture}[line cap=round,line join=round,x=1.2cm,y=1.2cm]
        \draw [line width=1.pt] (-0.3,0) -- (0.3,0);
        \draw [fill=black] (-0.3,0) circle (2pt);
        \draw [fill=black] (0.3,0) circle (2pt);
        \draw [line width=1.pt] (-0.3,0) .. controls (-1,0.7) and (-1, -0.7) .. (-0.3,0);
        \draw (0.3,0) node[anchor=west] {1};
        \draw (0,-1) node[anchor=north] {IV};
        \end{tikzpicture}
    \end{minipage}
    \begin{minipage}[t]{.12\textwidth}
        \begin{tikzpicture}[line cap=round,line join=round,x=1.2cm,y=1.2cm]
        \draw [fill=black] (0,0) circle (2pt);
        \draw [line width=1.pt] (0,0) .. controls (-0.8,0.8) and (-0.8, -0.8) .. (0,0);
        \draw (0,0) node[anchor=west] {1};
        \draw (0,-1) node[anchor=north] {V};
        \end{tikzpicture}
    \end{minipage}
    \begin{minipage}[t]{.12\textwidth}
        \begin{tikzpicture}[line cap=round,line join=round,x=1.2cm,y=1.2cm]
        \draw [line width=1.pt] (-0.3,0) -- (0.3,0);
        \draw [fill=black] (-0.3,0) circle (2pt);
        \draw [fill=black] (0.3,0) circle (2pt);
        \draw (-0.3,0) node[anchor=east] {1};
        \draw (0.3,0) node[anchor=west] {1};
        \draw (0,-1) node[anchor=north] {VI};
        \end{tikzpicture}
    \end{minipage}
    \begin{minipage}[t]{.12\textwidth}
        \begin{tikzpicture}[line cap=round,line join=round,x=1.2cm,y=1.2cm]
        \draw [fill=black] (0,0) circle (2pt);
        \draw (0,0) node[anchor=west] {2};
        \draw (0,-1) node[anchor=north] {VII};
        \end{tikzpicture}
    \end{minipage}
    \caption{The seven unmarked types in $\ob(\Gamma_{2,n})$.}
    \label{comboTypes}
\end{figure}

\begin{de}
A graph $(G,m,w) \in \Gamma_{2,n}$ has \textbf{theta type} if it satisfies the following conditions:
\begin{itemize}
    \item its unmarked type is the theta graph, i.e., type I in Figure \ref{comboTypes};
    \item its marking function $m$ is injective.
\end{itemize} The graph is said to have \textbf{cyclic theta type} if all the markings lie on a single cycle and \textbf{full theta type} otherwise.
\label{def: theta type graph}
\end{de}

\begin{rmk}
\cite[Lemma 4.2]{ChanTopOfM2n} states that the objects in $\Gamma_{2,n}$ are the disjoint union of the bridge locus and graphs of theta type.
\end{rmk}

\begin{example}[$S_2$-equivariant homology of $\Delta_{2,2}$] By \cite[Table 1]{ChanTopOfM2n}, the only nontrivial reduced rational homology of $\Delta_{2,n}$ for $n\leq 3$ is $\tilde{H}_4(\Delta_{2,2};\Q)$. In this example, we compute by hand $\tilde{H}_4(\Delta_{2,2};\Q)$ as an $S_2$-representation using the contractibility of the bridge locus. We show that $\tilde{H}_4(\Delta_{2,2};\Q)$ is isomorphic to the trivial representation.

We use the relative cellular chain complex associated with the pair $(\Delta_{2,n}, \Delta_{2,n}^{\text{br}})$. There are four combinatorial types in $\Gamma_{2,2}$ that are theta type, as defined in Definition \ref{def: theta type graph}, shown in Figure \ref{delta22}. However, $T_2,T_3$ and $T_4$ all have odd automorphisms, so they are zero in their respective chain groups. Therefore, the only chain group that has a nontrivial element is $C_4(\Delta_{2,2},\Delta_{2,2}^\text{br};\Q) = \langle [T_1,\tau] \rangle$, where $\tau$ is any edge-labeling on $T_1$. Hence $\tilde{H}_4(\Delta_{2,2};\Q) \cong C_4(\Delta_{2,2},\Delta_{2,2}^\text{br};\Q)$. The permutation $(12) \in S_2$ induces an even automorphism on $T_1$, so $(12)$ fixes $[T_1,\tau]$, thus $S_2$ acts trivially. We conclude that $\tilde{H}_4(\Delta_{2,2};\Q) \cong \chi_2$ as $S_2$-representations.

\begin{figure}
    \centering
    \begin{minipage}[t]{.21\textwidth}
        \begin{tikzpicture}[line cap=round,line join=round,x=1cm,y=1cm]
        \draw [line width=1.pt] (-0.5,0) -- (0.5,0);
        \draw [fill=black] (-0.5,0) circle (1pt);
        \draw [fill=black] (0.5,0) circle (1pt);
        \draw [line width=1.pt] (-0.5,0) .. controls (-0.4,0.8) and (0.4, 0.8) .. (0.5,0);
        \draw [line width=1.pt] (-0.5,0) .. controls (-0.4,-0.8) and (0.4, -0.8) .. (0.5,0);
        \draw [fill=black] (0,0.6) circle (1pt);
        \draw [fill=black] (0,0) circle (1pt);
        \draw [line width=1.pt] (0,0.6) -- (0,0.9);
        \draw (0,0.9) node[anchor=south] {1};
        \draw [line width=1.pt] (0,0) -- (0,0.3);
        \draw (0,0.3) node[anchor= west] {2};
        
        \draw (0,-1) node[anchor= north] {$T_1$};
        \end{tikzpicture}
    \end{minipage}
    \begin{minipage}[t]{.21\textwidth}
        \begin{tikzpicture}[line cap=round,line join=round,x=1cm,y=1cm]
        \draw [line width=1.pt] (-0.5,0) -- (0.5,0);
        \draw [fill=black] (-0.5,0) circle (1pt);
        \draw [fill=black] (0.5,0) circle (1pt);
        \draw [line width=1.pt] (-0.5,0) .. controls (-0.4,0.8) and (0.4, 0.8) .. (0.5,0);
        \draw [line width=1.pt] (-0.5,0) .. controls (-0.4,-0.8) and (0.4, -0.8) .. (0.5,0);
        \draw [fill=black] (0,0.6) circle (1pt);
        \draw [line width=1.pt] (0,0.6) -- (0,0.9);
        \draw (0,0.9) node[anchor=south] {1};
        \draw [line width=1.pt] (0.5,0) -- (0.8,0);
        \draw (0.8,0) node[anchor= west] {2};
        
        \draw (0,-1) node[anchor= north] {$T_2$};
        \end{tikzpicture}
    \end{minipage}
    \begin{minipage}[t]{.21\textwidth}
        \begin{tikzpicture}[line cap=round,line join=round,x=1cm,y=1cm]
        \draw [line width=1.pt] (-0.5,0) -- (0.5,0);
        \draw [fill=black] (-0.5,0) circle (1pt);
        \draw [fill=black] (0.5,0) circle (1pt);
        \draw [line width=1.pt] (-0.5,0) .. controls (-0.4,0.8) and (0.4, 0.8) .. (0.5,0);
        \draw [line width=1.pt] (-0.5,0) .. controls (-0.4,-0.8) and (0.4, -0.8) .. (0.5,0);
        \draw [fill=black] (0,0.6) circle (1pt);
        \draw [line width=1.pt] (0,0.6) -- (0,0.9);
        \draw (0,0.9) node[anchor=south] {2};
        \draw [line width=1.pt] (0.5,0) -- (0.8,0);
        \draw (0.8,0) node[anchor= west] {1};
        
        \draw (0,-1) node[anchor= north] {$T_3$};
        \end{tikzpicture}
    \end{minipage}
    \begin{minipage}[t]{.21\textwidth}
        \begin{tikzpicture}[line cap=round,line join=round,x=1cm,y=1cm]
        \draw [line width=1.pt] (-0.5,0) -- (0.5,0);
        \draw [fill=black] (-0.5,0) circle (1pt);
        \draw [fill=black] (0.5,0) circle (1pt);
        \draw [line width=1.pt] (-0.5,0) .. controls (-0.4,0.8) and (0.4, 0.8) .. (0.5,0);
        \draw [line width=1.pt] (-0.5,0) .. controls (-0.4,-0.8) and (0.4, -0.8) .. (0.5,0);
        \draw [line width=1.pt] (-0.8,0) -- (-0.5,0);
        \draw (-0.8,0) node[anchor=east] {1};
        \draw [line width=1.pt] (0.5,0) -- (0.8,0);
        \draw (0.8,0) node[anchor= west] {2};
        
        \draw (0,-1) node[anchor= north] {$T_4$};
        \end{tikzpicture}
    \end{minipage}
    \caption{Four combinatorial types in $\ob(\Gamma_{2,2})$ whose unmarked type is theta.}
    \label{delta22}
\end{figure}
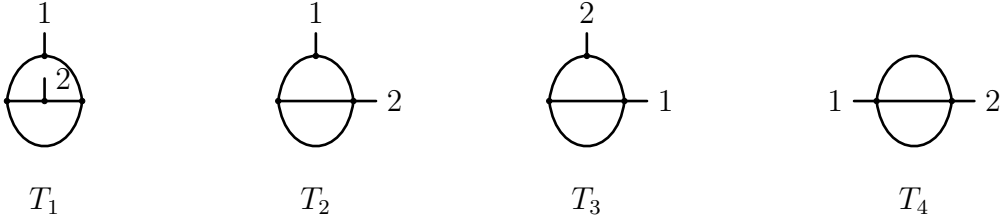

\end{example}

For $n\geq 4$, we recall another subcomplex of $\Delta_{2,n}$ that has zero reduced rational homology. Let $\theta^\circ$ denote the union of the bridge locus and graphs of cyclic theta type in $\ob(\Gamma_{2,n})$. Let $\Delta_{2,n}^{\theta^\circ}$ be the subcomplex generated by $\theta^\circ$.

\begin{thm}[\cite{ChanTopOfM2n}]
For $n\geq 4$, the reduced integral homology of $\Delta_{2,n}^{\theta^\circ}$ is \[\tilde{H}_i(\Delta_{2,n}^{\theta^\circ};\Z) = \begin{cases} 
(\Z/2\Z)^{\frac{(n-1)!}{2}} & i=n+1 \\
0 & \text{else}\end{cases}.\]
\label{cyclicTheta}
\end{thm}

\noindent Therefore $\tilde{H}_i(\Delta_{2,n}^{\theta^\circ};\Q)$ is trivial. The next proposition establishes a further simplification in the computation of $\tilde{H}_i(\Delta_{2,n};\Q)$.\\

\begin{prop}
There is an $S_n$-equivariant isomorphism $\tilde{H}_i(\Delta_{2,n};\Q) \cong H_i(C_*(\Delta_{2,n}, \Delta_{2,n}^{\theta^\circ};\Q))$ where $C_*(\Delta_{2,n}, \Delta_{2,n}^{\theta^\circ};\Q)$ is the relative cellular chain complex of the pair $(\Delta_{2,n}, \Delta_{2,n}^{\theta^\circ})$ of symmetric $\Delta$-complexes.
\label{prop: H(Delta2n)=H(Delta2n, cyclic theta locus)}
\end{prop}

\begin{proof}
We obtain that $\tilde{H}_i(\Delta_{2,n};\Q) \cong H_i(\Delta_{2,n}, \Delta_{2,n}^{\theta^\circ};\Q)$ by considering the long exact sequence associated to the pair. Furthermore, $\Delta_{2,n}([-1])$ is the unique combinatorial type in $\Gamma_{2,n}$ that has one vertex and $\Delta^{\theta^\circ}_{2,n}([-1]) = \Delta_{2,n}([-1])$. So by \cite[Proposition 3.6]{CGPTopOfMwithMarkedPoints}, there is an isomorphism $H_i(\Delta_{2,n},\Delta_{2,n}^{\theta^\circ};\Q) \cong H_i(C_*(\Delta_{2,n}, \Delta_{2,n}^{\theta^\circ}; \Q))$. Both isomorphisms are $S_n$-equivariant.
\label{homIsomorphism}
\end{proof}

\begin{rmk}
The cellular chain group $C_p(\Delta_{2,n}, \Delta_{2,n}^{\theta^\circ};\Q)$ has a basis $\mathcal{B} = \{[\G,\tau]: \G \in \ob(\Gamma_{2,n}), |E(\G)|=p+1, \G \text{ has full theta type, and } \tau:E(\G)\to [p]$ is any fixed edge labeling of $\G\}$. These chain groups serve as inputs to our computation.
\label{remark: relChainGroupBasis}
\end{rmk}
\noindent The cellular chain complex of the pair $(\Delta_{2,n}, \Delta_{2,n}^{\theta^\circ})$ turns out to have very few nonzero terms, as seen in the next proposition.
\begin{lem}\cite{ChanTopOfM2n}
$C_*(\Delta_{2,n}, \Delta_{2,n}^{\theta^\circ};\Q)$ is nontrivial in only three degrees: $n+2, n+1$, and $n$.
\end{lem}

\begin{proof}
By Remark \ref{remark: relChainGroupBasis}, it suffices to show that graphs of full theta type can have only $n+3$, $n+2$, or $n+1$ edges.

All graphs of theta type can be obtained from adding a marking function $m$ to the theta graph $\Theta$ that has two distinct vertices and three edges between them in the following way: for $i \in [n]$, either let $m(i)$ be an existing vertex in $\Theta$, or subdivide an existing edge to create a new vertex $v_{\text{new}}$ and let $m(i) = v_{\text{new}}$. Since $\Theta$ has two vertices and three edges, a graph of theta type has $n+1$, $n+2$ or $n+3$ edges.

\end{proof}

\begin{rmk}
So the chain complex $C_*(\Delta_{2,n}, \Delta_{2,n}^{\theta^\circ};\Q)$ is simplified to \[0 \to C_{n+2}(\Delta_{2,n}, \Delta_{2,n}^{\theta^\circ};\Q) \xrightarrow{d_{n+2}} C_{n+1}(\Delta_{2,n}, \Delta_{2,n}^{\theta^\circ};\Q) \xrightarrow{d_{n+1}} C_n(\Delta_{2,n}, \Delta_{2,n}^{\theta^\circ};\Q) \to 0.\]
\label{shortChainComplex}
\end{rmk}

\noindent Moreover, \cite{ChanTopOfM2n} shows the following proposition.

\begin{prop}
The boundary map $d_{n+1}$ in the chain complex in Remark \ref{shortChainComplex} is surjective. Hence the rational homology of $\Delta_{2,n}$ is only nontrivial in degrees $n+1$ and $n+2$.
\end{prop}

\subsection{The homology groups $\tilde{H}_i(\Delta_{2,n};\Q)$ as $S_n$-representations}

The symmetric group $S_n$ acts on the rational homology $\tilde{H}_i(\Delta_{2,n};\Q)$ by permuting markings and thus induces the structure of $S_n$-representations. By Proposition \ref{homIsomorphism}, we can calculate the characters of these representations by considering the $S_n$-action on the relative cellular chain groups $C_*(\Delta_{2,n}, \Delta_{2,n}^{\theta^\circ};\Q)$. In this section, we discuss our computation of these characters.

We only compute the character $\chi(H_{n+2}(\Delta_{2,n}, \Delta_{2,n}^{\theta^\circ};\Q))$ for simplicity. The top degree homology group $H_{n+2}(\Delta_{2,n}, \Delta_{2,n}^{\theta^\circ};\Q) = \ker d_{n+2}$ is a subrepresentation of the relative chain group $C_{n+2}(\Delta_{2,n}, \Delta_{2,n}^{\theta^\circ};\Q)$, and the latter has an explicit basis given combinatorially. We then use Theorem 1.1 in \cite{CFGP} to derive $\chi(\tilde{H}_{n+1}(\Delta_{2,n};\Q))$. For details of the derivation, see Remark \ref{discussionOfFarber}.


Our method relies on the following facts from representation theory for finite groups. Let $G$ be a finite group. Let $\rho: G \to \gl(V)$ and $\psi: G \to \gl(U)$ be two linear representations on the finite dimensional vector spaces $V$ and $U$ over $\C$. 

The vector spaces $V$ and $U$ can be canonically decomposed into isotypic components. More precisely, let $W_1,\dots,W_h$ be the irreducible representations of $G$. Then we can uniquely write $V = V_1 \oplus \cdots \oplus V_h$ and $U = U_1\oplus \cdots \oplus U_h$, where $V_i \cong W_i^{\oplus c_i}$ and $U_i \cong W_i^{\oplus d_i}$. 

Let $f: V \to U$ be a $G$-equivariant morphism of vector spaces. We are interested in computing $\ker f$ as a subrepresentation of $V$. First, we observe that Schur's Lemma implies that $\ker f = \oplus_{i=1}^h \ker f_i$, where $f_i = f|_{V_i}$. Then we state a proposition that calculates $\ker f_i$ as a subrepresentation of $V_i$. This proposition is a adaptation of \cite[Chapter 2, Proposition 8]{serre} for our computations. 

We first introduce some notations. Let the irreducible $G$-representation $W_i$ be given in matrix form $(r_{\alpha\beta}(s))$ for $s\in G$ with respect to a basis $(e_1,\dots,e_n)$. Define a linear map $p_{\alpha\beta}: V\to V$ as follows: 
\begin{equation}
    p_{\alpha\beta} = \frac{\dim W_i}{|G|}\sum_{g\in G} r_{\beta\alpha}(g^{-1})\rho(g),
    \label{equation: projection to the first drawer}
\end{equation} for $1\leq \alpha,\beta \leq n$. Let $V_{i,\alpha}$ be the image of $p_{\alpha\alpha}$.

\begin{prop}
    Let $\{x_1,\dots,x_m\}$ be a basis of $V_{i,1}$. Then $\ker f_i \cong W_i^{\oplus k}$ where $k = c_i - \rk \{f(x_1),\dots,f(x_m)\}$.
\end{prop}

\begin{proof}
    Consider $V_i$ and $U_i$. We may further decompose them into irreducible $G$-representations. Let $V_i = V_i^1 \oplus \cdots \oplus V_i^{c_i}$ and $U_i = U_i^1 \oplus \cdots \oplus U_i^{d_i}$, where $V_i^j \cong U_i^k \cong W_i$ for all $j,k$. Fix a basis for each $V_i^j$ and $U_i^k$ and write $f_i$ in matrix form with respect to these bases. Then we have \[
    f_i = \begin{blockarray}{ccccc}
    V_i^1 & V_i^2 & \cdots & V_i^{c_i} \\
    \begin{block}{(cccc)c}
        0 & \cdots & & 0 & U_1\\
        \vdots & & & & \vdots \\
        0 & \cdots & & 0 & U_{i-1}\\
        F_i^{11} & F_i^{12} & \cdots & F_i^{1c_i} & U_i^1 \\
        \vdots & & & & \vdots \\
        F_i^{d_i1} & F_i^{d_i2} & \cdots & F_i^{d_ic_i} & U_i^{d_i} \\
        0 & \cdots & & 0 & U_{i+1}\\
        \vdots & & & & \vdots\\
        0 & \cdots & & 0 & U_h \\
    \end{block}
    \end{blockarray},\] where $F_i^{jk}$ and 0 represent block matrices.By Schur's Lemma, each $F_i^{jk}$ is zero or an isomorphism. With appropriate choices of bases for $U_i^j$, $F_i^{jk}$ is either the identity matrix or the zero matrix. Therefore, if we write $V_{i(1)}^j$ as the first column in the $V_i^j$ block, then $\rk f_i = \rk \{V_{i(1)}^j\}_{i=1}^{c_i} \cdot \dim W_i$. By the rank-nullity theorem, $\dim \ker f_i = \dim V_i - \rk f_i$. Since $\ker f_i$ is a subrepresentation of $V_i$, we know that $\ker f_i \cong W_i^{k}$ for some $k$. We can calculate $k = \dim \ker f_i/\dim W_i = c_i - \rk \{V_{i(1)}^j\}_{i=1}^{c_i}$.
    
    If the function $f$ is given in matrix form with respect to arbitrary bases, then we need to produce a basis for each $V_i^j$ and perform a change of basis. By \cite[Proposition 8]{serre}, $p_{\alpha\alpha}$ is zero on $V_j$ for $j\neq i$ and is a projection onto its image. Moreover, its image $V_{i,\alpha}$ is contained in $V_i$ and has dimension $c_i$. Given a basis $\{x_1, \dots, x_{c_i}\}$ for $V_{i,1}$, we can compute a basis for each $V_i^j$ as $\{p_{\alpha1}(x_j)\}_{\alpha=1}^n$. Therefore, after we change to this basis of $V_i$, we have $V_{i(1)}^j = f(x_j)$.
\end{proof}

In our computation, we take $V = C_{n+2}(\Delta_{2,n}, \Delta_{2,n}^{\theta^\circ};\Q)$, $U = C_{n+1}(\Delta_{2,n}, \Delta_{2,n}^{\theta^\circ};\Q)$ and $f = d_{n+2}$. Our computer program follows the steps below.
\begin{enumerate}
    \item We manually compute the character of $C_{n+2}(\Delta_{2,n}, \Delta_{2,n}^{\theta^\circ};\Q)$ and decompose it into irreducibles. Let $C_{n+2}(\Delta_{2,n}, \Delta_{2,n}^{\theta^\circ};\Q) \cong \oplus V_{\lambda_i}$, where $V_{\lambda_i} \cong \oplus S_{\lambda_i}^{c_i}$. Here $S_{\lambda_i}$ denotes the Specht module indexed by the partition $\lambda_i$.
    \item We generate bases for $C_{n+2}(\Delta_{2,n}, \Delta_{2,n}^{\theta^\circ};\Q)$ and $C_{n+1}(\Delta_{2,n}, \Delta_{2,n}^{\theta^\circ};\Q)$. Then we generate $d_{n+2}$ as a matrix with respect to these bases. This step took 224 seconds for $n=8$.
    \item For each $V_{\lambda_i}$, we generate a basis $B_{\lambda_i,1}$ for $V_{\lambda_i,1}$ by computing the images of $c_i$ random vectors in $V$ under the projection $p_{11}^i$ and check that they do form a basis. We use the Sage built-in matrix representations for Specht modules when producing $p_{11}^i$. This step took between 3.7 hours to 13.9 hours depending on $\lambda_i$ for $n=8$.
    \item We compute the rank of $\{d_{n+2} \circ v\}_{v \in B_{\lambda_i,1}}$ and hence obtain $\ker d_{n+2}|_{V_{\lambda_i,1}}$. This step took up to 450 seconds depending on the value of $c_i$.
\end{enumerate}
The computation for the case $n=8$ took 162 hours in total. The results of these computations are presented in Theorem \ref{mainThm}. The most expensive step is computing vectors in the image of $p_{11}^i$ as it requires summing over the entire symmetric group. Code is available at the author's website. \footnote{\url{https://www.math.brown.edu/~claudia_yun/}}

\begin{rmk}
\label{discussionOfFarber}
\cite[Theorem 1.1]{CFGP} gives the generating function $z_g$ for the $S_n$-equivariant top weight Euler characteristic of the moduli space $\mathcal{M}_{g,n}$ as a symmetric function. By the identification of Equation \eqref{Top weight cohomology of Mgn is identified with homology of Deltagn}, $z_g$ is also the generating function for the $S_n$-equivariant Euler characteristics of $\Delta_{g,n}$ after a change of sign. It is related to our computation in the following way. First, we specialize the formula to the case $g=2$. This is computed in \cite[Example 8.3]{CFGP}. 
\begin{equation}
    z_2 = -\frac{1}{12}\frac{1}{P_1} + \frac{1}{2}\frac{P_1}{P_2}-\frac{1}{6}\frac{P_1^2}{P_3}-\frac{1}{12}\frac{P_1^3}{P_2^2}-\frac{1}{6}\frac{P_2P_3}{P_6},
    \label{z2}
\end{equation} where $P_i=1+p_i$ and $p_i = \sum_{j>0}x_j^i$. We then focus on its degree $n$ part, denoted by $(z_2)_n$. By the formula after Definition 2.1 in \cite{CFGP}, 
\begin{equation}
    (z_2)_n = (-)\sum_{i\geq 0} \sum_{\sigma \in S_n} \frac{(-1)^i}{n!}\chi_{V^n_i}(\sigma)\psi(\sigma),
    \label{degNpart}
\end{equation} where $\chi_{V^n_i}$ is the character of the representation afforded by $\tilde{H}_{i}(\Delta_{2,n};\Q)$. The symmetric function $\psi(\sigma)$ is defined in the following way: suppose $\sigma \in S_n$ has cycle type $a_1\geq a_2\geq \cdots \geq a_l$. Then \[\psi(\sigma) = p_{a_1}p_{a_2}\cdots p_{a_l}.\] 
Since $\chi_{V^n_i}$ is only nontrivial when $i=n+1$ and $n+2$, the right-hand side of Equation \eqref{degNpart} simplifies. For a conjugacy class $C(\sigma)$ in $S_n$ represented by $\sigma$, the coefficient of $\psi(\sigma)$ in $(z_2)_n$ is 
\begin{equation}
    (z_2)_{n,\psi(\sigma)} = \frac{|C(\sigma)|}{n!}[(-1)^{n}\chi_{\tilde{H}_{n+1}(\Delta_{2,n};\Q)}(\sigma) + (-1)^{n+1}\chi_{\tilde{H}_{n+2}(\Delta_{2,n};\Q)}(\sigma)].
    \label{anotherWayToCalculateZ2}
\end{equation} So the difference between the characters of the top degree and lower degree homology representations can be read off of $z_2$. For $n\leq 6$, our computations agree with this formula; for $n=7$, we computed $\tilde{H}_9(\Delta_{2,7};\Q)$ in SageMath and combined Equation \eqref{z2} and Equation \eqref{anotherWayToCalculateZ2} to derive $\tilde{H}_8(\Delta_{2,7};\Q)$.
\label{faber}
\end{rmk}

\subsection{Analysis of $\tilde{H}_7(\Delta_{2,5};\Q)$}
Elements in the rational homology groups of $\Delta_{2,n}$ are rather mysterious. We hope that by analyzing the combinatorics of theta graphs we might have a better understanding of the structure of these homology groups. In this section, we identify and study a subrepresentation of $\tilde{H}_7(\Delta_{2,5};\Q)$ that has an explicit combinatorial description. 

By decomposing the character of $\tilde{H}_7(\Delta_{2,5};\Q)$, we know that $\tilde{H}_7(\Delta_{2,5};\Q) \cong S^{311} \oplus S^{32} \oplus S^{41}$, where $S^p$ is the Specht module indexed by a partition $p \vdash 5$. Using the projection formula described by Serre in \cite[Section 2.6]{serre}, we can compute the vector subspaces affording each irreducible subrepresentaton. The vector subspace $V_{311} \subseteq \tilde{H}_7(\Delta_{2,5};\Q)$ isormorphic to $S^{311}$ turns out to have a particularly nice basis. So we focus on $V_{311}$, give a somewhat intuitive basis, and establish an explicit representation isomorphism from $V_{311} \to S^{311}$. For simplicity, write $\tilde{H}_7(\Delta_{2,5};\Q)$ as $H$.

We first describe a basis of $V_{311}$. Let \[C = \left\{g_1 = \thetaThreeOneOne{4}{2}{3}{0}{1}, \quad g_2 = -\thetaThreeOneOne{2}{3}{1}{0}{4}, \quad g_3 = \thetaTwoTwoOne{2}{3}{4}{1}{0}, \quad g_4 = \thetaTwoTwoOne{2}{4}{3}{1}{0} \right\} \subset C_7\left(\Delta_{2,5},\Delta_{2,5}^{\theta^\circ};\Q\right),\] where the edge-labeling is 0 through 7 in reading order starting from the upper left edge for each theta graph. The edge-labeling is understood to be this one unless otherwise noted.

Let $\pi = (01234) \in S_5$. Let $\overline{C} = \bigcup_{i\geq 0} \pi^i C$, where $\pi C = \{[(G,m \circ \pi,w),\tau] : (G,m,w)\in C \}$. One can check that $\pi^i C \cap \pi^j C = \emptyset$ if $i \not \equiv j \mod 5$, treating $\pi^i C$ as subsets of $\ob(\Gamma_{2,5})$ when we forget the edge-labels.

\begin{prop}
    The element $v = \sum_{\G \in \overline{C}} \G$ is in $\tilde{H}_7(\Delta_{2,5};\Q) \cong H_7(\Delta_{2,5}, \Delta_{2,5}^{\theta^\circ};\Q)$ as in Proposition \ref{prop: H(Delta2n)=H(Delta2n, cyclic theta locus)}.
\end{prop}

\begin{proof}
Consider the cellular chain complex \[0 \to C_7\left(\Delta_{2,5}, \Delta_{2,5}^{\theta^\circ};\Q\right) \xrightarrow{\partial} C_6\left(\Delta_{2,5}, \Delta_{2,5}^{\theta^\circ};\Q\right) \to C_5\left(\Delta_{2,5}, \Delta_{2,5}^{\theta^\circ};\Q\right) \to 0.\] Since $\tilde{H}_7(\Delta_{2,5}, \Delta_{2,5}^{\theta^\circ};\Q) \cong \ker \partial$, it suffices to show that $\partial(v) = 0$.

Let $u = \sum_{\G\in C}\G$. Then \[\partial(u) = - \thetaTwoOneOne{2}{4}{0}{3}{1} + \thetaTwoOneOne{4}{1}{2}{0}{3} - \thetaTwoOneOne{1}{4}{0}{3}{2}  + \thetaTwoOneOne{3}{1}{2}{0}{4}.\] Let the first summand be $h$ and the third summand be $k$. Notice that the second summand is $\pi^2h$ and the fourth is $\pi^2 k$, so we can rewrite \[\partial(u) = -h+\pi^2h-k+\pi^2k.\] Since $\partial$ is $S_5$-equivariant, $\partial(\sum_{\G \in \pi^iC}\G) = \partial(\pi^iu) = \pi^i\partial (u) =-\pi^ih+\pi^{i+2}h-\pi^ik+\pi^{i+2}k$. So 
        \begin{align*}
            \partial(v) &= \partial \left(\sum_{i=0}^4 \pi^iu\right)\\
            &= \sum_{i=0}^4\pi^i\partial(u)\\
            &= 0
        \end{align*}
\end{proof}

\begin{prop}
Let $V_{311} \subseteq H$ be the vector subspace isomorphic to the Specht module $S^{311}$ as representations. Then $v \in V_{311}$.
\end{prop}

\begin{proof}
By decomposing the character, we know that \[H \cong S^{311} \oplus S^{32} \oplus S^{41}\] as $S_5$-representation. Since each irreducible subrepresentation has multiplicity 1, there is a unique decomposition of $H$ into irreducible $S_5$-invariant subspaces. Let $\sigma \vdash 5$. The projection map on $H$ to its irreducible subspace isomorphic to $S^\sigma$ is \[P_\sigma = \frac{\chi_\sigma(e)}{|S_5|}\sum_{\pi \in S_5}\chi_\sigma(\pi)^*\cdot\pi,\] where $\pi$ is understood as an automorphism on $H$ and $*$ denotes the complex conjugate \cite[Section 2.6]{serre}. In this case $\chi_\sigma$ is always real, so $\chi_\sigma(\pi)^* = \chi_\sigma(\pi)$. In sage, we compute the projection to $V_{311}$ to be
\setcounter{MaxMatrixCols}{15}
\[P_{311} = \begin{pmatrix}
10 & 3 & 3 & -6 & -1 & -1 & 3 & -5 & -2 & 3 & -5 & -2 & 1 & 1 & 2 \\
2 & 9 & 1 & 4 & -3 & -1 & -5 & 1 & 0 & 1 & -3 & -6 & 1 & 3 & 0 \\
0 & 0 & 8 & 2 & -2 & -4 & 2 & -2 & -4 & -4 & 2 & 2 & 4 & 2 & 4 \\
-4 & 2 & 2 & 8 & 0 & 0 & -4 & 2 & -2 & -4 & 2 & -2 & 0 & 0 & 2 \\
4 & -3 & -1 & 2 & 9 & 1 & 1 & -5 & 0 & -3 & 1 & -6 & -1 & -9 & 0 \\
2 & -2 & -4 & 0 & 0 & 8 & -2 & 2 & -4 & 2 & -4 & 2 & -8 & 0 & 4 \\
2 & -4 & 2 & -4 & 2 & -2 & 8 & 0 & 0 & -4 & -2 & 2 & 2 & -2 & 0 \\
-4 & 3 & -3 & 2 & -3 & 1 & 3 & 11 & 2 & -5 & -7 & 0 & -1 & 3 & -2 \\
-2 & 2 & -4 & -2 & 2 & -4 & 0 & 0 & 8 & 2 & 2 & -4 & 4 & -2 & -8 \\
2 & 2 & -4 & -4 & -2 & 2 & -4 & -2 & 2 & 8 & 0 & 0 & -2 & 2 & -2 \\
-4 & -3 & 3 & 2 & 1 & -3 & -5 & -7 & 0 & 3 & 11 & 2 & 3 & -1 & 0 \\
-4 & -5 & 1 & -4 & -5 & 1 & 3 & 3 & -2 & 1 & 1 & 10 & -1 & 5 & 2 \\
2 & 4 & 2 & -2 & -2 & -4 & 4 & 2 & 2 & -2 & -4 & -2 & 4 & 2 & -2 \\
2 & 2 & 4 & -2 & -4 & -2 & -2 & -4 & -2 & 4 & 2 & 2 & 2 & 4 & 2 \\
4 & 2 & 2 & 4 & 2 & 2 & -2 & -2 & -4 & -2 & -2 & -4 & -2 & -2 & 4
\end{pmatrix},\] with respect to the basis described in Remark \ref{remark: relChainGroupBasis}. We check in Sage that $v$ is in the image of $P_{311}$.
\end{proof}

\begin{prop}
    The set $S=\{\sigma v\}_{\sigma \in S_{\{0,1,2\}}}$ is a basis for $V_{311}$. We also compute the unique (up to rescaling) isomorphism from $V_{311} \to S^{311}$ in this basis.
\end{prop}

\begin{proof}
Let $\sigma \in S_{\{0,1,2\}} < S_5$. Since $V_{311}\subset H$ is $S_5$-invariant, $\sigma v \in V_{311}$. Denote $v$ as $v_1$ for convenience. Write $v_2 = (01)v, v_3 = (02)v, v_4 = (12)v, v_5 = (012)v$, and $v_6 = (021)v$. Consider the following summands in $v_1$, \[\thetaTwoTwoOne{1}{2}{3}{0}{4}+\thetaTwoTwoOne{1}{3}{2}{0}{4}+\thetaTwoTwoOne{0}{1}{2}{4}{3}+\thetaTwoTwoOne{0}{2}{1}{4}{3}+\cdots.\] Then consider the orbits of these summands under $S_{\{0,1,2\}}$ .
\begin{align*}
    v_2 &= \thetaTwoTwoOne{0}{2}{3}{1}{4}+\thetaTwoTwoOne{0}{3}{2}{1}{4}+\thetaTwoTwoOne{1}{0}{2}{4}{3}+\thetaTwoTwoOne{1}{2}{0}{4}{3}+\cdots \\
    v_3 &= \thetaTwoTwoOne{1}{0}{3}{2}{4}+\thetaTwoTwoOne{1}{3}{0}{2}{4}+\thetaTwoTwoOne{2}{1}{0}{4}{3}+\thetaTwoTwoOne{2}{0}{1}{4}{3}+\cdots
\end{align*}
\begin{align*}
    v_4 &= \thetaTwoTwoOne{2}{1}{3}{0}{4}+\thetaTwoTwoOne{2}{3}{1}{0}{4}+\thetaTwoTwoOne{0}{2}{1}{4}{3}+\thetaTwoTwoOne{0}{1}{2}{4}{3}+\cdots \\
    v_5 &= \thetaTwoTwoOne{2}{0}{3}{1}{4}+\thetaTwoTwoOne{2}{3}{0}{1}{4}+\thetaTwoTwoOne{1}{2}{0}{4}{3}+\thetaTwoTwoOne{1}{0}{2}{4}{3}+\cdots \\
    v_6 &= \thetaTwoTwoOne{0}{1}{3}{2}{4}+\thetaTwoTwoOne{0}{3}{1}{2}{4}+\thetaTwoTwoOne{2}{0}{1}{4}{3}+\thetaTwoTwoOne{2}{1}{0}{4}{3}+\cdots
\end{align*} Now let \[\sum_{i=1}^6 c_iv_i = 0,\] where $c_i \in \Q$. Comparing the first two summands in each vector, we see that it must be true that $c_1 = -c_2$, $c_3 = -c_5$, and $c_4 = -c_6$. Comparing the last two summands, we conclude that $c_1 = c_4$, $c_2 = c_5$ and $c_3 = c_6$. Therefore, \[c_1 = -c_2 = -c_5 = c_3 = c_6 = -c_4 = -c_1,\] so $c_i = 0$ for all $i$, thus $S$ is linearly independent. Notice that $S^{311}$ has dimension 6, so $S$ is a basis of $V_{311}$.

Let $W=\{e_{t_1}, e_{t_2}, e_{t_3}, e_{t_4}, e_{t_5}, e_{t_6}\}$ be a basis for $S^{311}$, where $e_{t_i}$ is a standard polytabloid as described in \cite{sagan}. Let
\ytableausetup{tabloids,centertableaux}
\begin{align*}
    t_1 &= \ytableaushort{012,3,4}, \quad t_2 = \ytableaushort{013,2,4}, \quad t_3 = \ytableaushort{014,2,3} \\
    \\
    t_4 &= \ytableaushort{023,1,4}, \quad t_5 = \ytableaushort{024,1,3}, \quad t_6 = \ytableaushort{034,1,2}.
\end{align*} Using $S$ as a basis for $V_{311}$ and $W$ as a basis for $S^{311}$, we can construct an explicit isomorphism between them. Let $h:V_{311} \to S^{311}$ be the identity matrix. Let $\rho^1:S_5\to \gl(V_{311})$ and $\rho^2:S_5\to \gl(S^{311})$. Then \[h^0 = \sum_{\pi \in S_5}(\rho^2(\pi))^{-1}h\rho^1(\pi)\] is $S_5$-equivariant and is either 0 or an isomorphism by Schur's lemma \cite{serre}. We calculate $h^0$ in Sage to be \[h^0 = 20\cdot \begin{pmatrix}
1 & 1 & 1 & 1 & 1 & 1 \\
1 & 1 & -1 & -1 & -1 & -1 \\
1 & 1 & -1 & 1 & 1 & -1 \\
-1 & -1 & -1 & 1 & -1 & 1 \\
1 & -1 & 1 & 1 & -1 & 1 \\
1 & 1 & -1 & -1 & 1 & -1
\end{pmatrix}.\]
\end{proof}

\noindent The code for this section can also be found on the author's website.

\begin{rmk}
We conclude with a dicussion of some open questions and natural extensions of the results of this paper.
\begin{itemize}
    \item \cite[Theorem 1.1]{CFGP} gives the generating function for the difference of the characters $\chi_{\tilde{H}_{n+1}(\Delta_{2,n};\Q)}$ and $\chi_{\tilde{H}_{n+2}(\Delta_{2,n};\Q)}$, as discussed in Remark \ref{faber}. Can we find generating functions for $\chi_{\tilde{H}_{n+1}(\Delta_{2,n};\Q)}$ and $\chi_{\tilde{H}_{n+2}(\Delta_{2,n};\Q)}$ individually in a similar fashion using symmetric functions? This might be challenging. We see in Remark \ref{faber} that $(z_2)_n$ only contains $\psi(\sigma)$ for $\sigma$ having cycles of lengths 1, 2, 3, and 6. This is possible when we consider the difference between two characters as their values may coincide for many conjugacy classes. However, when we consider the individual characters $\chi_{\tilde{H}_{n+1}(\Delta_{2,n};\Q)}$ and $\chi_{\tilde{H}_{n+2}(\Delta_{2,n};\Q)}$, we already see from our results that they have nonzero characters on elements with various cycle types for each $n$. In particular, they seem to have nonzero value on cycle types with cycles of arbitrary lengths. That poses a challenge to representing their characters as symmetric functions.
    \item From the discussion of this section, we see that these representations in $V_{311}$ have interesting combinatorics. Is there a way to generalize the combinatorial structure we saw to higher $n$'s? Can we use the combinatorics of theta graphs to produce a nonzero element in $\tilde{H}_{n+2}(\Delta_{2,n}, \Delta^{\theta^\circ}_{2,n};\Q)$ for each $n$?
    \item $\Delta_{2,n}$ has nontrivial reduced rational homology in the top two degrees for $n\geq 4$. That means $\Delta_{2,n}$ can no longer be bouquet of spheres like the cases when $g=0$ and $g=1$. For $g=2$, by manual construction and Remark \ref{rmk: bridge locus is contractible}, we know that $\Delta_{2,0}$ and $\Delta_{2,1}$ are contractible. What can we say about the homotopy types of $\Delta_{2,n}$ for $n\geq 2$?
\end{itemize} 
\end{rmk}

\bibliographystyle{amsalpha}
\bibliography{main}

\newpage
\appendix

\section{Character tables}
\label{charTable}

\begin{table}[h]
\centering
\begin{tabular}{|l|l|l|l|l|l|}
\hline
Conjugacy class & 1111 & 211 & 22 & 31 & 4 \\ \hline
$H_6(\Delta_{2,4};\Q)$ & 3   & -1   & -1       & 0     & 1  \\ \hline
$H_5(\Delta_{2,4};\Q)$ & 1   & 1    & 1        & 1     & 1   \\ \hline  
\end{tabular}
\caption{Character table for $n=4$.}
\label{char4}
\end{table}

\begin{table}[h]
\centering
\begin{tabular}{|l|l|l|l|l|l|l|l|}
\hline
Conjugacy class  & 11111 & 2111 & 221 & 311 & 32 & 41 & 5 \\ \hline
$H_7(\Delta_{2,5};\Q)$ & 15  & 3    & -1       & 0     & 0         & -1     & 0       \\ \hline
$H_6(\Delta_{2,5};\Q)$ & 5   & 1    & 1        & -1    & 1         & -1     & 0       \\ \hline
\end{tabular}
\caption{Character table for $n=5$.}
\label{char5}
\end{table}

\begin{table}[h]
\centering
\begin{tabular}{|l|l|l|l|l|l|l|l|l|l|l|l|}
\hline
Conjugacy class & 111111 & 21111 & 2211 & 222 & 3111 & 321 & 33 & 411 & 42 & 51 & 6 \\ \hline
$H_8(\Delta_{2,6};\Q)$ & 86& 2& 10& 6& -1& -1& 2 & 0 & 0 & 1 & 0 \\ \hline
$H_7(\Delta_{2,6};\Q)$ & 26& 2& -2& -2& -1& -1& -1 & 0 & 0 & 1 & 1 \\ \hline
\end{tabular}
\caption{Character table for $n=6$.}
\label{char6}
\end{table}

\begin{table}[h]
\centering
\begin{tabular}{ |l|l|l|l|l|l|l|l|l|l|l|}
\hline
Conjugacy class & 1111111 & 211111 & 22111 & 2221 & 31111 & 3211 & 322 & 331 & 4111 & 421  \\ \hline
    $H_9(\Delta_{2,7};\Q)$ &575&5&-13&17&-1&-1&-1&-1&-1&1\\ \hline
    $H_8(\Delta_{2,7};\Q)$ &155&5&-1&-7&-1&-1&-1&5&-1&1\\ \hline
\end{tabular}
\begin{tabular}{|l|l|l|l|l|l| }
\hline
    Conjugacy class & 43 & 511 & 52 & 61 & 7 \\ \hline
    $H_9(\Delta_{2,7};\Q)$ &-1&0&0&-1&1 \\ \hline
    $H_8(\Delta_{2,7};\Q)$ &-1&0&0&-1&1\\ \hline
\end{tabular}
\caption{Character table for $n=7$.}
\label{char7}
\end{table}

\begin{table}[h!]
\centering
\begin{tabular}{ |l|l|l|l|l|l|l|l|l|l|}
\hline
   Conjugacy class & 11111111 & 2111111 & 221111 & 22211 & 2222 & 311111 & 32111 & 3221 & 3311  \\ \hline
    $H_{10}(\Delta_{2,8};\Q)$ & 4426 & 16 & -2 & -84 & -30 & 1 & 1 & 1 & 4\\ \hline
    $H_9(\Delta_{2,8};\Q)$ & 1066 & 16 & -2 & 12 & 2 & 1 & 1 & 1 & -2\\ \hline
\end{tabular}
\begin{tabular}{|l|l|l|l|l|l|l|l|l|l|l|l|l|l| }
\hline
    Conjugacy class &332 & 41111 & 4211 & 422 & 431 & 44 & 5111 & 521 & 53 & 611 & 62 & 71 & 8 \\ \hline
    $H_{10}(\Delta_{2,8};\Q)$ & 4 & -2 & 0 & 2 & 1 & -2 & 1 & 1 & 1 & 0 & 0 & 2 & 0 \\ \hline
    $H_9(\Delta_{2,8};\Q)$ & 4 & -2 & 0 & 2 & 1 & -2 & 1 & 1 & 1 & 0 & 2 & 2 & 0 \\ \hline
\end{tabular}
\caption{Character table for $n=8$.}
\label{table: char8}
\end{table}
\end{document}